\documentclass[12pt]{article}
\usepackage{amsmath, amsthm, amssymb}
\usepackage{enumerate}

\newtheorem{theorem}{Theorem}[section]
\newtheorem{lemma}[theorem]{Lemma}
\newtheorem{proposition}[theorem]{Proposition}
\newtheorem{corollary}[theorem]{Corollary}
\newtheorem{question}[theorem]{Question}

\newtheorem*{theorem:universal}{Theorem~\ref{theorem:universal}}
\newtheorem*{theorem:nilpotent}{Theorem~\ref{theorem:nilpotent}}
\newtheorem*{corollary:nilpotent}{Corollary~\ref{corollary:nilpotent}}
\newtheorem*{corollary:nilpotent_lift}{Corollary~\ref{corollary:nilpotent_lift}}
\newtheorem*{corollary:abelian}{Corollary~\ref{corollary:abelian}}
\newtheorem*{theorem:polynomial}{Theorem~\ref{theorem:polynomial}}
\newtheorem*{theorem:coherence}{Theorem~\ref{theorem:coherence}}
\newtheorem*{theorem:dc3}{Theorem~\ref{theorem:dc3}}
\newtheorem*{theorem:quasi-isometry}{Theorem~\ref{theorem:quasi-isometry}}

\theoremstyle{definition}
\newtheorem{definition}[theorem]{Definition}
\newtheorem{example}[theorem]{Example}

\theoremstyle{remark}
\newtheorem{remark}[theorem]{Remark}
\numberwithin{equation}{section}

\newcommand{\F}{{\cal F}}
\newcommand{\p}{{\parallel}}
\newcommand{\abs}[1]{|  #1  |}

\begin{document}
\title{On $3$-manifolds that support partially hyperbolic diffeomorphisms \footnotetext{$2000$ \textit{Mathematics Subject Classification.} Primary 34C40, 37D10, 37D30.} }
\author{Kamlesh Parwani}
\date{\today}

\maketitle

\begin{abstract}
Let $M$ be a closed $3$-manifold that supports a partially hyperbolic diffeomorphism $f$. If $\pi_1(M)$ is nilpotent, the induced action of $f_*$ on $H_1(M, \mathbb{R})$ is partially hyperbolic. If $\pi_1(M)$ is almost nilpotent or if $\pi_1(M)$ has subexponential growth, $M$ is finitely covered by a circle bundle over the torus. If $\pi_1(M)$ is almost solvable, $M$ is finitely covered by a torus bundle over the circle. Furthermore, there exist infinitely many hyperbolic $3$-manifolds that do not support dynamically coherent partially hyperbolic diffeomorphisms; this list includes the Weeks manifold.

If $f$ is a strong partially hyperbolic diffeomorphism on a closed $3$-manifold $M$ and if $\pi_1(M)$ is nilpotent, then the lifts of the stable and unstable foliations are quasi-isometric in the universal of $M.$ It then follows that $f$ is dynamically coherent.

We also provide a sufficient condition for dynamical coherence in any dimension. If $f$ is center bunched and if  the center-stable and center-unstable distributions are Lipschitz, then the partially hyperbolic diffeomorphism $f$ must be dynamically coherent.
\end{abstract}

\section{Introduction}
There are many definitions of partial hyperbolicity; we will use the one given below.
\begin{definition} \label{definition:phype}
A $C^1$ diffeomorphism $f$ on a compact manifold $M$ is \textit{partially hyperbolic} if the following conditions hold. There is a nontrivial splitting of the tangent bundle, $TM = E^s \oplus E^c \oplus E^u$, that is invariant under the derivative map $Df.$ There exists a Riemannian metric for which we can choose continuous positive functions $\nu$, $\hat{\nu}$, $\gamma$, and $\hat{\gamma}$ such that 
\begin{equation*}
\nu, \hat{\nu} < 1 \textrm{ and } \nu < \gamma < \hat{\gamma}^{-1} < \hat{\nu}^{-1}.
\end{equation*}
Furthermore, for every unit vector $v \in T_p M$,
\begin{align*}
\p Df(v) \p  < \nu(p), &  \; \; \textrm{if } v \in E^{s}(p), \\
  \\
\gamma(p)  <  \p Df(v) \p < \hat{\gamma}^{-1}(p), &  \; \; \textrm{if } v \in E^{c}(p), \\
  \\
\hat{\nu}(p)^{-1} <  \p Df(v) \p , &  \; \; \textrm{if } v \in E^{u}(p).
\end{align*}
\end{definition}

Some authors assume that the functions $\nu$, $\hat{\nu}$, $\gamma$, and $\hat{\gamma}$ can be chosen to be constants---this is sometimes referred to as \textit{strong partial hyperbolicity} or uniform partial hyperbolicity.

In general, $E^s$, $E^u$, and $E^c$ are not smooth distributions; they are always H\"{o}lder.  The distributions $E^s$ and $E^u$ are uniquely integrable, but $E^c$ need not be integrable. However, if $E^{cs}$ and $E^{cu}$ are both integrable, then $E^c$ is also integrable.

\begin{definition}
We say that $f$ is \textit{dynamically coherent} if there exist $f$-invariant, $C^0$ foliations  $\F^{cs}$ and $\F^{cu}$ tangent to $E^{cs}=E^s \oplus  E^c$ and $E^{cu}= E^u \oplus E^c$ respectively.
\end{definition}

\begin{question}
Is every partially hyperbolic diffeomorphism on $3$-manifold dynamically coherent?
\end{question}

There has been some progress in showing that dynamical coherence is natural to partially hyperbolic maps on $3$-manifolds.  For example, Bonatti and Wilkinson have shown that if a transitive strong partially hyperbolic diffeomorphism possesses an invariant circle tangent to the center distribution, then the map must be dynamically coherent (see \cite{BoWi}). Recently, Brin, Burago, and Ivanov in \cite{BBI2} established dynamical coherence for strong partially hyperbolic diffeomorphisms on $T^3.$ We prove the same result for strong partially hyperbolic diffeomorphisms on $3$-manifolds with nilpotent fundamental groups. It should be noted that there are no known examples of strong partially hyperbolic diffeomorphisms on $3$-manifolds that are not dynamically coherent.  However, there are examples of partially hyperbolic diffeomorphisms on six-dimensional manifolds that are not dynamically coherent.

The theorems in this article are motivated by the following question.

\begin{question}
Which $3$-manifolds support partially hyperbolic diffeomorphisms?
\end{question}

\subsection{Examples}

We now present a few examples of partially hyperbolic diffeomorphisms on $3$-manifolds.  The first few are described in greater detail in \cite{KH}. Example~\ref{example:skew} will be referred to again in Section~\ref{section:dc3}.

\begin{example}[Anosov Diffeomorphisms] \label{example:Anosov}
The most ``famous'' examples of partially hyperbolic maps are Anosov diffeomorphisms on the three torus $T^3.$ Other well-known examples of partially hyperbolic maps that are not Anosov are certain Derived from Anosov diffeomorphisms.
\end{example}

\begin{example}[Suspensions] \label{example:suspension}
Choose an Anosov diffeomorphism on the torus $T^2$ and then obtain an Anosov flow on a Sol manifold via the standard suspension construction.  The time-one map of the Anosov flow is a dynamically coherent partially hyperbolic diffeomorphism. Note the fundamental group of such a $3$-manifold is solvable and the manifold is a torus bundle over the circle.
\end{example}

\begin{example}[Anosov flows] \label{example:flow}
The time-one map of any Anosov flow on a $3$-manifold is a dynamically coherent partially hyperbolic diffeomorphism.  Some common examples of Anosov flows are suspensions of Anosov diffeomorphisms and geodesic flows on unit tangent bundles of surfaces with constant negative curvature.  More exotic examples can be obtained via surgery techniques (see \cite{Goodman} and \cite{Handel&Thurston}).
\end{example}

So the problem of classifying $3$-manifolds that support partially hyperbolic diffeomorphisms is at least as intractable as the problem of classifying $3$-manifolds that support Anosov flows.

\begin{example}[Skew products] \label{example:skew}
There is a simple construction of a dynamically coherent partially hyperbolic diffeomorphism on the 3-dimensional torus $T^3.$  If $A$ is a linear Anosov diffeomorphism on $T^2$, define a partially hyperbolic diffeomorphism on $T^3$ by $f(v,z) = (Av, z + \theta)$, where $v \in T^2$, $z \in S^1$, and $\theta \in S^1$ is a fixed angle.

This construction on $T^3$ can be generalized to any orientable circle bundle over the torus with Euler class $k$, denoted by $M_k$, in the following manner (see \cite{BoWi} where this construction is referred to as a topological skew product over an Anosov diffeomorphism). Any orientation preserving diffeomorphism $f : T^2 \to T^2$ can be expressed as a composition of 2 diffeomorphisms: $f = g_2 \circ g_1$, where $g_i$ is the identity on a neighborhood $U_i$ of a compact disk $D_i \subset T^2.$ For each $i$, the circle bundle $\pi_k : \pi_k^{-1} (T^2 - int (D_i )) \to (T^2 - int (D_i ))$ is trivializable by a fibered chart $\psi_i$ that induces rotations in the fibers.  This is because the isomorphism classes of orientable circle bundles over a space are in one-to-one correspondence with the second integral cohomology group of the space, and $H^2(T^2 - int (D_i ), \mathbb{Z})$ is trivial. Now define a diffeomorphism $G_i$ of $\pi_k^{-1} (T^2 - int (D_i ))$ as follows: $G_i$ coincides with the identity map on $\pi_k^{-1}(U_i )$ and, in the trivialization given by the chart $\psi_i$, coincides with $(g_i , id_{S^ 1} ).$ 
The map $F = G_2 \circ G_1$ is a partially hyperbolic diffeomorphism on $M_k$ when $g_2 \circ g_1$ is an Anosov diffeomorphism.

The manifolds obtained from this construction have nilpotent fundamental groups---they are Nil manifolds.  These circle bundles over the torus are also torus bundles over the circle.
\end{example}

\subsection{Statement of results}

\begin{definition}
A group is \textit{almost solvable} if it contains a finite-index normal subgroup that is solvable. A group is \textit{almost nilpotent} if it contains a finite-index normal subgroup that is nilpotent. \textit{Almost abelian} groups are defined similarly.
\end{definition}

We can classify all closed $3$-manifolds that support partially hyperbolic diffeomorphisms and have almost solvable and almost nilpotent fundamental groups.

\begin{theorem} \label{theorem:solvable/nilpotent}
Let $M$ be a closed $3$-manifold that supports a partially hyperbolic diffeomorphism.
\begin{enumerate}[\normalfont 1)]
\item If $\pi_1(M)$ is solvable, $M$ is finitely covered a by torus bundle over the circle.
\item If $\pi_1(M)$ is nilpotent, $M$ is a circle bundle over the torus.
\end{enumerate}
\end{theorem}

\begin{corollary} \label{corollary:solvable/nilpotent}
Let $M$ be a closed $3$-manifold that supports a partially hyperbolic diffeomorphism.
\begin{enumerate}[\normalfont 1)]
\item If $\pi_1(M)$ is almost solvable, $M$ is finitely covered by a torus bundle over the circle.
\item If $\pi_1(M)$ is almost nilpotent, $M$ is finitely covered by a circle bundle over the torus.
\end{enumerate}
\end{corollary}

More can be said about the partially hyperbolic map if $\pi_1(M)$ is nilpotent.

\begin{theorem} \label{theorem:nilpotent}
Let $M$ be a closed $3$-manifold that supports a partially hyperbolic diffeomorphism $f$. If $\pi_1(M)$ is nilpotent, the induced action $f_*$ on $H_1(M,\mathbb{R})$ is partially hyperbolic, that is, $f_*$ has eigenvalues $\lambda^u$  and $\lambda^s$ such that $\abs{\lambda^u} > 1$ and  $\abs{\lambda^s} < 1.$
\end{theorem}

Note that the definition of partial hyperbolicity for the induced action on homology is relaxed to include cases in which the center direction may not even exist. Also observe that if $\pi_1(M)$ is solvable, then the induced action on $H_1(M, \mathbb{R})$ can be trivial. A time-one map of an Anosov flow from Example~\ref{example:suspension} is isotopic to the identity, and therefore, has a trivial induced action on $H_1(M, \mathbb{R}).$ So the induced action on  $H_1(M, \mathbb{R})$ need not be partially hyperbolic when $\pi_1(M)$ is solvable.

These theorems are inspired by the recent work of Brin, Burago, and Ivanov, who showed that if $M$ supports a strong partially hyperbolic diffeomorphism and $\pi_1(M)$ is abelian, then the induced action on $H_1 (M, \mathbb{R})$ is partially hyperbolic (there is an eigenvalue with absolute value greater than one and another eigenvalue with absolute value less than one).

We assume that the Geometrization Conjecture holds for the next result (see \cite{Perelman1}--\cite{Perelman3} and \cite{Morgan-Tian}).

\begin{theorem} \label{theorem:polynomial}
Let $M$ be a closed $3$-manifold that supports a partially hyperbolic diffeomorphism $f$. If $\pi_1(M)$ has subexponetial growth, then $M$ is finitely covered by a circle bundle over the torus.
\end{theorem}

The following result is an immediate corollary to the results in \cite{BI}. In particular, it implies that there are no dynamically coherent partially hyperbolic diffeomorphisms on most lens spaces and reducible $3$-manifolds (like $T^3 \# T^3$).

\begin{theorem} \label{theorem:universal}
If $M$ is a closed $3$-manifold that supports a partially hyperbolic diffeomorphism, then the universal cover of $M$ is homeomorphic to $\mathbb{R}^3.$ In particular, $\pi_2(M)=0$ and $\pi_1(M)$ is infinite.
\end{theorem}

There are infinitely many $3$-manifolds that have universal covers homeomorphic to $\mathbb{R}^3$ and do not support partially hyperbolic diffeomorphisms.

\begin{theorem} \label{theorem:hyperbolic}
There exist infinitely many hyperbolic $3$-manifolds that do not support dynamically coherent partially hyperbolic diffeomorphisms. This list includes the Weeks manifold.
\end{theorem}

The paper is organized in the following manner. In Section~\ref{section:universal} we show that Theorem~\ref{theorem:universal} follows easily from the recent results in \cite{BI}.  In Section~\ref{section:hyperbolic}, we prove Theorem~\ref{theorem:hyperbolic}. Theorem~\ref{theorem:solvable/nilpotent}, Corollary~\ref{corollary:solvable/nilpotent}, and Theorem~\ref{theorem:polynomial} are established in Section~\ref{section:solvable/nilpotent}. Section~\ref{section:homology} is devoted to the proof of Theorem~\ref{theorem:nilpotent}. In Section~\ref{section:dc3} and Section~\ref{section:coherence} we provide sufficient conditions for dynamical coherence.

\begin{theorem} \label{theorem:quasi-isometry}
Let $f$ be a strong partially hyperbolic diffeomorphism on a closed $3$-manifold $M.$ If $\pi_1(M)$ is nilpotent, the lifts of the unstable and stable foliations are quasi-isometric in the universal cover.
\end{theorem}

\begin{theorem}\label{theorem:dc3}
Let $f$ be a strong partially hyperbolic diffeomorphism on a closed $3$-manifold $M.$ If $\pi_1(M)$ is nilpotent, $f$ is dynamically coherent.
\end{theorem}

The next result is not restricted to dimension three and applies whenever the partially hyperbolic diffeomorphism is center bunched (see Section~\ref{section:coherence} for more details).

\begin{theorem} \label{theorem:coherence}
Let $f$ be a partially hyperbolic diffeomorphism. If $f$ is center bunched and the distributions $E^{cs}$ and $E^{cu}$ are Lipschitz, then $f$ is dynamically coherent.
\end{theorem}

\noindent
Sections \ref{section:universal} and \ref{section:solvable/nilpotent} contain preliminary background information and results; the material in these sections is purely topological. The main theorems of this article are contained in Section~\ref{section:homology} and Section~\ref{section:dc3}; these sections employ arguments that use dynamics. Section~\ref{section:coherence} is mostly self-contained and discusses results that are not restricted to dimension three. Section~\ref{section:hyperbolic} is a short digression into the area of essential laminations on hyperbolic $3$-manifolds. We end the paper with several open questions in Section~\ref{section:questions}.

The author would like to thank K.~Burns, R.~Saghin, and the referees for providing many suggestions and comments that have considerably improved the exposition in this article.
\section{Codimension one foliations on $3$-manifolds} \label{section:universal}

We will first present some classical results from the rich theory of codimension one foliations on $3$-manifolds. The books \cite{Fol1} and \cite{Fol2} are excellent references for all things pertaining to foliations.

\subsection{Reeb components and spherical leaves} 
Let $M$ be a closed $3$-manifold that supports a $C^0$, codimension one, orientable, and transversely orientable foliation $\F.$ The main scoundrel, or obstruction to the presence of a ``nice" structure for $\F$ and $M$, is a Reeb component.

\begin{definition}
Let $D^2$ be the closed unit disc and let $\mathbb{R}^{+}=\{x \in \mathbb{R} : x \geq 0 \}.$ Consider the submersion $g: D^2 \times \mathbb{R} \to \mathbb{R}^{+}$ defined by $g(r,\theta,z) = (1 - r^2) e^z.$ Let $g^{-1}(c)$ for $c>0$ be the leaves of the foliation $\widetilde{\cal{R}}$ on $D^2 \times \mathbb{R}.$ Now quotient by the covering translation $z \to z+1$ to obtain the famous Reeb foliation $\cal{R}$ on $D^2 \times S^1.$ The boundary cylinder of $D^2 \times S^1$ together with the Reeb foliation is the standard Reeb component. All \textit{Reeb components} are homeomorphic to this standard Reeb component.
\end{definition}

The other obstruction is a spherical leaf.  If one can show that there are no Reeb components and spherical leaves in $\F$, then the universal cover of $M$ is homeomorphic to $\mathbb{R}^3.$  This result is a direct consequence of some very famous theorems of Novikov and Palmeira.

\begin{theorem}[Novikov \cite{Novikov}] \label{theorem:Novikov}
Suppose that $\F$ contains no Reeb components.
\begin{enumerate}[\normalfont 1)]
\item If $\pi_2(M)$ is nontrivial, $\F$ contains a spherical leaf.
\item If $\pi_2(M)$ is trivial, then the lift of every leaf to the universal cover of $M$ is homeomorphic to a plane.
\end{enumerate}
\end{theorem}

This theorem was initially established for $C^2$ foliations in \cite{Novikov}. Later, Solodov in \cite{Solodov} generalized this result by proving the same theorem for $C^0$ foliations. In fact, Solodov remarks that the proof would be considerably easier if there existed a line field transverse to the foliation.  Note that Siebenmann has shown that this is always the case (see \cite{Siebenmann}).

\begin{theorem}[Palmeira \cite{Palmeira}] \label{theorem:Palmeira}
If $\widetilde{M}$ is a simply connected manifold foliated by codimension one leaves homeomorphic to $\mathbb{R}^{n-1}$, then $\widetilde{M}$ is homeomorphic to $\mathbb{R}^n.$
\end{theorem}

Again, this result was initially proved for $C^2$ foliations.  However, the existence of a transverse line field to a foliation is sufficient to prove the same result for $C^0$ foliations.

A direct consequence of Theorem~\ref{theorem:Novikov} and Theorem~\ref{theorem:Palmeira} is the following result.

\begin{corollary} \label{corollary:universal}
If $\pi_2(M)$ is trivial and if $\F$ contains no Reeb components, then the universal cover of $M$ is homeomorphic to $\mathbb{R}^3.$ 
\end{corollary}

\begin{definition}
Let $M$ be a manifold that supports a codimension one foliation $\F.$ A \textit{contractible cycle} is a closed, null-homotopic curve that is transverse to the foliation $\F.$
\end{definition}

The following result will be used repeatedly.

\begin{theorem}[Novikov \cite{Novikov}] \label{theorem:Reeb}
Let $M$ be a closed $3$-manifold with a codimension one foliation $\F$ that is orientable and transversely orientable. If there exists a contractible cycle, then $\F$ contains a Reeb component.
\end{theorem}

\subsection{The universal cover is $\mathbb{R}^3$}

Let $M$ be a closed $3$-manifold that supports a partially hyperbolic diffeomorphism $f$. After lifting to a finite cover $\overline{M}$ if necessary, assume that the distributions $E^{s}$, $E^{u}$, $E^{cs}$ and $E^{cu}$ are orientable; we abuse notation here and use the same names for the lifts of these distributions. Under these conditions, Burago and Ivanov have demonstrated the existence of two codimension one foliations $\F^{cs}_{\epsilon}$ and $\F^{cu}_{\epsilon}$ that are ``almost tangent" to $E^{cs}$ and $E^{cu}$ respectively.

\begin{theorem}[Burago and Ivanov \cite{BI}] \label{theorem:epsilon}
For every $\epsilon > 0$, there exist  $C^0$ foliations  $\F^{cs}_{\epsilon}$ and $\F^{cu}_{\epsilon}$ with $C^1$ leaves such that the angles between $T \F^{cs}_{\epsilon}$ and $E^{cs}$ and the angles between $T \F^{cu}_{\epsilon}$ and $E^{cu}$ are no greater than $\epsilon.$ If $\epsilon$  is sufficiently small, the foliations $\F^{cs}_{\epsilon}$ and $\F^{cu}_{\epsilon}$ are Reebless.
\end{theorem}

In \cite{BI} it is also shown that there exists a continuous map $h_\epsilon : \overline{M} \to \overline{M}$ such that $\textrm{dist}_{C^0}(h_\epsilon, id_M) < \epsilon$ and $h_\epsilon$ sends every leaf in $\F^{cs}_{\epsilon}$ to a surface tangent to $E^{cs}$ (there is another continuous map with identical properties that sends leaves in $\F^{cu}_{\epsilon}$ to surfaces tangent to $E^{cu}$). 

It is not too difficult to see that the foliations $\F^{cs}_{\epsilon}$ and $\F^{cu}_{\epsilon}$ have no spherical leaves.

\begin{lemma} \label{lemma:no_spheres}
The foliations $\F^{cs}_{\epsilon}$ and $\F^{cu}_{\epsilon}$ do not contain spherical leaves when $\epsilon$ is sufficiently small.
\end{lemma}
\begin{proof}
Suppose that $\F^{cs}_{\epsilon}$ contains a spherical leaf. The Reeb Stability Theorem implies that $M$ is homeomorphic to $S^2 \times S^1$ and $\F^{cs}_{\epsilon}$ is the (product) foliation by spheres on $\overline{M}$, the appropriate finite cover. Let $\F^u$ be the unstable foliation associated with $f.$ Now $\epsilon$ can be chosen sufficiently small so that the lift of $\F^u$ to $\overline{M}$ is transverse to all the spherical leaves in $\F^{cs}_{\epsilon}.$ So $\F^u$ must contain a closed leaf, since the return map to any sphere must have a fixed point. However, the unstable foliation $\F^u$ does not contain any closed leaves. A similar argument shows that $\F^{cu}_{\epsilon}$ does not contain spherical leaves also when $\epsilon$ is sufficiently small.
\end{proof}

So an immediate corollary to the results in \cite{BI} is that the universal cover of $M$ is homeomorphic to $\mathbb{R}^3.$  This result is not explicitly stated in \cite{BI}.

\begin{theorem:universal}
If $M$ is a closed $3$-manifold that supports a partially hyperbolic diffeomorphism, then the universal cover of $M$ is homeomorphic to $\mathbb{R}^3.$ In particular, $\pi_2(M)=0$ and $\pi_1(M)$ is infinite.
\end{theorem:universal}
\begin{proof}
Consider the codimension one foliation $\F^{cs}_{\epsilon}$ on $\overline{M}$, the appropriate finite cover of $M.$ Assume that $\epsilon$ is small enough to insure that there are no Reeb components in $\F^{cs}_{\epsilon}$ and the lift of $\F^u$ to $\overline{M}$ is transverse to $\F^{cs}_{\epsilon}.$ If $\pi_2(\overline{M})$ is nontrivial, Theorem~\ref{theorem:Novikov} implies the existence of a spherical leaf. According to Lemma~\ref{lemma:no_spheres}, $\F^{cs}_{\epsilon}$ contains no spherical leaves. Therefore, $\pi_2(\overline{M})$ is trivial, and now the theorem follows from Corollary~\ref{corollary:universal}.
\end{proof}

\section{Hyperbolic $3$-manifolds} \label{section:hyperbolic}

In this section we observe that recent results about  foliations on hyperbolic $3$-manifolds imply that there are infinitely many hyperbolic $3$-manifolds that do not support dynamically coherent partially hyperbolic diffeomorphisms. The theorems stated below immediately imply Theorem~\ref{theorem:hyperbolic}.

\begin{theorem}[Fenley \cite{Fenley_laminar}]
There exist infinitely many hyperbolic $3$-manifolds that do not support essential laminations.
\end{theorem}

The author considers an infinite family of $3$-manifolds $\cal{M}$ obtained by doing Dehn surgery along a closed curve in a torus bundle over the circle. The remarkable result established is that every lamination, and hence every foliation,  on $M \in \cal{M}$ must contain a torus leaf bounding a solid torus.  The interested reader should also see \cite{RSS} for a similar result.

Now, any manifold in the infinite family $\cal{M}$ cannot support dynamically coherent partially hyperbolic diffeomorphisms since the foliations $\F^{cs}$ and $\F^{cu}$ must be Reebless.

\begin{theorem}[Calegari and Dunfield \cite{Calegari&Dunfield}]
The Weeks manifold does not support any Reebless foliations.
\end{theorem}

This result implies that there are no dynamically coherent partially hyperbolic diffeomorphisms on the Weeks manifold.

\begin{remark}
The assumption of dynamical coherence in Theorem~\ref{theorem:hyperbolic} can be dropped if the distributions $E^{cs}$ and $E^{u}$ are orientable and transversely orientable. In this case, we simply use $\F^{cs}_{\epsilon}$, with $\epsilon$ sufficiently small to insure that $\F^{cs}_{\epsilon}$ is Reebless. The problem here is that Brin, Burago, and Ivanov prove the existence of $\F^{cs}_{\epsilon}$ under the assumption of orientability and transverse orientability.  To do this, they lift to the appropriate finite cover and it is not clear if $\F^{cs}_{\epsilon}$ on this finite cover descends to a foliation on the original manifold.
\end{remark}

\begin{remark}
There are infinitely many hyperbolic $3$-manifolds that do support dynamically coherent partially hyperbolic diffeomorphisms. Surgery constructions on standard examples of Anosov flows give rise to Anosov flows on infinitely many hyperbolic $3$-manifolds (see \cite{Goodman}), and the time-one maps of these flows are partially hyperbolic and dynamically coherent.
\end{remark}

The problem of classifying partially hyperbolic diffeomorphisms on hyperbolic $3$-manifolds remains open. Is every partially hyperbolic diffeomorphism on a hyperbolic $3$-manifold a perturbation of some time-one map of an Anosov flow? More generally, we may ask the following question.

\begin{question}
Does there exist a $3$-manifold with exponential growth in its fundamental group that supports a partially hyperbolic diffeomorphism but does not support an Anosov flow?
\end{question}

\section{Almost solvable/nilpotent fundamental groups} \label{section:solvable/nilpotent}

Theorem~\ref{theorem:solvable/nilpotent} and Corollary~\ref{corollary:solvable/nilpotent} follow from the classification of $3$-manifolds with almost solvable and almost nilpotent fundamental groups in \cite{Evans&Moser}.

\begin{theorem}[Evans and Moser]
Let $M$ be a closed $3$-manifold and suppose that $\pi_2(M) = 0$ and $\pi_1(M)$ is infinite.
\begin{enumerate}[\normalfont 1)]
\item If $\pi_1(M)$ is solvable, $M$ is finitely covered by a torus bundle over the circle. 
\item If $\pi_1(M)$ is nilpotent, $M$ is a circle bundle over the torus.
\end{enumerate}
\end{theorem}

Moreover, if $\pi_1(M)$ is nilpotent, $M$ is a torus bundle over the circle and the monodromy map has the form $\left( \begin{array}{cc} 1 & n \\ 0 & 1 \end{array} \right),$ where $n$ is an integer.

\begin{corollary}[Evans and Moser] 
Let $M$ be a closed $3$-manifold and suppose that $\pi_2(M) = 0$ and $\pi_1(M)$ is infinite.
\begin{enumerate}[\normalfont 1)]
\item If $\pi_1(M)$ is almost solvable, $M$ is finitely covered by a torus bundle over the circle. 
\item If $\pi_1(M)$ is almost nilpotent, $M$ is finitely covered by a circle bundle over the torus.
\end{enumerate}
\end{corollary}

If $M$ supports a partially hyperbolic diffeomorphism, $\pi_2(M) = 0$ and $\pi_1(M)$ is infinite since the universal cover of $M$ is homeomorphic to $\mathbb{R}^3$. Now Theorem~\ref{theorem:solvable/nilpotent} and Corollary~\ref{corollary:solvable/nilpotent} follow from the results above.

\begin{definition} \label{definition:polynomial}
Suppose that $\pi_1(M)$ is finitely generated and $\{g_1,g_2, \dots, g_k \}$ is a symmetric generating set. Let $B_n = \{ \gamma \in \pi_1(M) : \p \gamma \p \leq n\}$ where  $\p \, \p$ is the word length measured using the generators $\{g_1,g_2, \dots, g_k \}.$ Then $\pi_1(M)$ \textit{has polynomial growth} if for all positive integers $n$,
\begin{equation*}
| B_n | \leq P(n), 
\end{equation*}
where $P(x)$ is some fixed polynomial and $| B_n |$ is the number of elements in $B_n.$

The fundamental group of $M$ \textit{has exponential growth} if 
\[
\lim_{n \to \infty} \frac{\ln \abs{B_n}}{n} > 0.
\]

We say that $\pi_1(M)$ \textit{has subexponential growth} if 
\[
\lim_{n \to \infty} \frac{\ln \abs{B_n}}{n} =0.
\]

\end{definition}

Note that the limits above always exist and these definitions are independent of the choice of generators. Also, polynomial growth implies subexponential growth, but subexponetial growth does not imply polynomial growth. However, in the case of fundamental groups of $3$-manifolds, subexponential growth often means polynomial growth. The result below is ``folklore'' and the proof is probably well known to the experts.  However, the author could not find a reference, and so, a short proof is included for the sake of completeness. For the next proposition, we assume that the Geometrization Conjecture holds. The reader is referred to the excellent survey article \cite{Scott} for the necessary background on the Geometrization Conjecture. 

\begin{proposition} \label{proposition:polynomial}
Let $M$ be a closed, $P^2$-irreducible $3$-manifold. If $\pi_1(M)$ is infinite and has subexponential growth, $M$ is finitely covered by a circle bundle over the torus. In particular, $\pi_1(M)$ is almost nilpotent, and hence, $\pi_1(M)$ has polynomial growth.
\end{proposition}
\begin{proof}
Without loss of generality, assume that $M$ is orientable. We will first establish that $M$ is a Seifert fibered space. If $M$ is not Haken and not Seifert fibered, $M$ is hyperbolic since $\pi_1(M)$ is infinite. However, the assumption of subexponential growth in $\pi_1(M)$ implies that $M$ cannot be hyperbolic. So assume that $M$ is Haken. Now, the irreducibility assumption permits us to use the JSJ-Decomposition (see \cite{Johannson} and \cite{Jaco&Shalen}) to obtain a \textit{minimal} Seifert fibered space $\Sigma$ such that every component of $M - \mathring{\Sigma}$ is atoriodal and has hyperbolic geometry. This is impossible unless $M - \mathring{\Sigma}$ is a disjoint union of tori and each component of $\Sigma$ is Seifert fibered. Again, the assumptions on the manifold $M$ imply that no component of $\mathring{\Sigma}$ can have $H^2 \times \mathbb{R}$ or $\widetilde{\textrm{SL}}_2(\mathbb{R})$ geometry. In this case, either $M$ is Seifert fibered or $M$ has Sol geometry. Since $\pi_1(M)$ does not have exponential growth, $M$ cannot support Sol geometry. So $M$ is a Seifert fibered space.
It now follows that $M$ has either Euclidean or Nil geometry because $\pi_1(M)$ is infinite and has subexponential growth. All closed $3$-manifolds with Euclidean or Nil geometries are finitely covered by circle bundles over the torus (see \cite{Scott}).
\end{proof}

\begin{theorem:polynomial}
Let $M$ be a closed $3$-manifold that supports a partially hyperbolic diffeomorphism. If $\pi_1(M)$ has subexponential growth, then $M$ is finitely covered by a circle bundle over the torus.
\end{theorem:polynomial}
\begin{proof}
If $M$ supports a partially hyperbolic diffeomorphism, $M$ is irreducible and $\pi_1(M)$ is infinite, since the universal cover of $M$ is homeomorphic to $\mathbb{R}^3.$ Now if $\pi_1(M)$ has subexponential growth, Proposition~\ref{proposition:polynomial} implies that $M$ is finitely covered by a circle bundle over the torus.
\end{proof}

Note that the dependence of this result on Geometrization can be removed by assuming that $\pi_1(M)$ has polynomial growth since Gromov's result in \cite{Gromov} would then imply that $\pi_1(M)$ is almost nilpotent and the desired conclusion would follow from Corollary~\ref{corollary:solvable/nilpotent}.

\section{Induced action on homology} \label{section:homology}

Theorem~\ref{theorem:nilpotent} is a consequence of the rigidity of the Seifert fibration, which implies that any homeomorphism on a nontrivial circle bundle over the torus can be isotoped to a well-understood, simple map. Since isotopy preserves the induced action on homology, in the proof of Theorem~\ref{theorem:nilpotent} we will move between the partially hyperbolic map $f$ and the simple map (denoted by $f_1$) at our convenience. First, we present a result of Waldhausen that makes all this possible.

Waldhausen proved that for $3$-manifolds that are Haken (compact, orientable, irreducible $3$-manifolds that contain orientable, incompressible surfaces besides the sphere and the disk) and Seifert fibered (circle bundles over a $2$-dimensional orbifold), every isotopy class can be represented by a fiber preserving homeomorphism.

\begin{theorem}[Waldhausen \cite{Waldhausen} and \cite{Waldhausen2}]
Let $M$ be a closed, Haken $3$-manifold and let $f$ be homeomorphism of $M.$ 
If $M$ is a Seifert fibered space, $f$ is isotopic to a fiber preserving homeomorphism unless $M$ is covered by $T^3.$
\end{theorem}

If $M$ supports a partially hyperbolic diffeomorphism and $\pi_1(M)$ is nilpotent, $M$ is a circle bundle over the torus. When $M$ is a non-trivial circle bundle over the torus with Euler class $n$, the fundamental group of $M$ is generated be the ``horizontal'' generators $x$ and $y$, where $[x,y] = t^n$, $[x,t] = [y,t] = 1$, and $t$ is the generator of the circular fiber. It then follows that $H_1(M, \mathbb{R}) = \mathbb{R} \oplus \mathbb{R}$ and is generated by the horizontal cycles $x$ and $y.$

We are now ready to show that a partially hyperbolic map induces a partially hyperbolic action on $H_1(M,\mathbb{R})$ when $\pi_1(M)$ is nilpotent.

\begin{theorem:nilpotent}
Let $M$ be a closed $3$-manifold that supports a partially hyperbolic diffeomorphism $f$. If $\pi_1(M)$ is nilpotent, the induced action $f_*$ on $H_1(M,\mathbb{R})$ is partially hyperbolic, that is, $f_*$ has eigenvalues $\lambda^u$  and $\lambda^s$ such that $\abs{\lambda^u} > 1$ and  $\abs{\lambda^s} < 1.$
\end{theorem:nilpotent}
\begin{proof}
Let $f$ be a  partially hyperbolic diffeomorphism on $M$ and assume that $\pi_1(M)$ is nilpotent. Theorem~\ref{theorem:solvable/nilpotent} implies that $M$ is a circle bundle over the torus. If $M$ is $T^3,$ the proof is simpler and will be presented later; also, this result for $T^3$ has been demonstrated in \cite{BBI2}. So assume that $M$ is a nontrivial circle bundle over the torus. The action of $f_*$ on $H_1(M,\mathbb{R})$ is partially hyperbolic if and only if the action of a power of $f_*$ on $H_1(M,\mathbb{R})$ is partially hyperbolic. So, without loss of generality, assume that $f_*$ is orientation preserving. Waldhausen's Theorem implies that $f$ is isotopic to a fiber preserving homeomorphism $f_1.$ Since $f_1$ preserves the fiber, there is an induced homeomorphism on the base torus which we denote by $q: T^2 \to T^2.$  Also, let $\F^{cs}_{\epsilon}$ be the codimension one foliation almost tangent to the lift of $E^{cs}$ on the appropriate finite cover $\overline{M}.$ Here $\epsilon$ is small enough to insure that the $\F^{cs}_{\epsilon}$ is Reebless and that the unstable foliation of $f$ (denoted by $\F^u$) lifted to $\overline{M}$ is transverse to $\F^{cs}_{\epsilon}.$

Now a lift of $f$ to $\mathbb{R}^3$, denoted by $\widetilde{f}$, is isotopic to $\widetilde{f_1}$ (the lift of $f_1$), where $\widetilde{f_1}$ is a homeomorphism of $\mathbb{R}^3$ that preserves vertical lines. If the induced action $f_*$ on $H_1(M,\mathbb{R})$ is not partially hyperbolic, then the same can be said about $f_{1*}.$ Note that the action on homology for these maps is the same as the action of $q_*$ on homology. So $q$ is isotopic to a finite order map or a power of a Dehn twist. In this situation, the vertical lines in $\mathbb{R}^3$ move apart at a linear rate under the action of $\widetilde{f_1}.$ Since $f_1$ preserves the circular fiber, points move apart in the vertical direction at a rate that is at most linear under the action of $\widetilde{f}_1.$ So the diameter of any ball under the action of $\widetilde{f_1}$ grows at a linear rate. Since the isotopy from $\widetilde{f}$ to $\widetilde{f_1}$ moves points a bounded distance, the diameter of any ball under the action of $\widetilde{f}$ grows at a rate that is at most polynomial.  This implies that the volume of balls grows at a polynomial rate ($\pi_1(M)$ has polynomial growth of degree $4$). However, the unstable manifolds grow exponentially. This disparity in the growth rates will allow us to obtain a contradiction in the following manner. First observe that when $l$ is a segment of the unstable manifold in $\widetilde{\F}^u$ (the lift of $\F^u$), $\widetilde{f}^n(l)$ cannot limit on itself. If two points on $\widetilde{f}^n(l)$ are sufficiently close, we may close up a segment in $\widetilde{f}^n(l)$ and obtain a closed curve transverse to $\widetilde{\F}^{cs}_{\epsilon}$, the lift of  $\F^{cs}_{\epsilon}.$ Now this closed curve projects to a null-homotopic closed curve on $M$ that is transverse to $\F^{cs}_{\epsilon}.$ Novikov's theorem (Theorem~\ref{theorem:Reeb}) implies the existence of a Reeb component in $\F^{cs}_{\epsilon}.$ This contradicts the fact that $\F^{cs}_{\epsilon}$ is Reebless. So there exists a fixed $\delta > 0$ such that for any $n>0$, the cylindrical tube with axis $\widetilde{f}^n(l)$ and radius $\delta$ does not intersect itself. Now the volumes of these cylinders grow exponentially as $n \to \infty$ since the volume of any such cylinder is a constant times the length of $\widetilde{f}^n(l)$ and the length of $\widetilde{f}^n(l)$ grows exponentially as $n \to \infty.$ However, these cylinders are contained in balls with volumes that grow at a polynomial rate. For $n$ sufficiently large, we obtain a contradiction. This implies that the induced action $f_*$ on  $H_1(M,\mathbb{R})$ must be partially hyperbolic.

Now if $M$ is $T^3$ and if the induced action $f_*$ on $H_1(M,\mathbb{R})$ is not partially hyperbolic, then $f$ is isotopic a linear map $f_1$ such that the diameter of balls in $\mathbb{R}^3$ grows at a polynomial rate  under the action of $\widetilde{f}_1.$ Again, the diameter of balls under the action of $\widetilde{f}$ grow at a polynomial rate also and the same arguments above apply.  So in all cases, the induced action $f_*$ on $H_1(M,\mathbb{R})$ is partially hyperbolic.
\end{proof}

\begin{remark}
The proof above does not work when $\pi_1(M)$ is solvable. This is because if $\pi_1(M)$ is (virtually) nilpotent, $\pi_1(M)$ has polynomial growth (see \cite{Gromov}), and then, since the universal cover is quasi-isometric to the fundamental group, the universal cover has polynomial volume growth. However, when $\pi_1(M)$ is solvable, the fundamental group may have exponential growth. This is exactly what happens for the Sol manifolds that support Anosov flows, and hence, support dynamically coherent partially hyperbolic diffeomorphisms. For these partially hyperbolic diffeomorphisms (the time-one maps of Anosov flows), the volume of balls in the universal cover grows at an exponential rate and we do not get the contradiction that we obtain above for manifolds with nilpotent fundamental groups.
\end{remark}

\begin{remark}
The arguments in the proof of Theorem~\ref{theorem:nilpotent} can be used to show that there are no partially hyperbolic diffeomorphisms on non-trivial circle bundles over the Klein bottle.  Let $M$ be an orientable non-trivial circle bundle over the Klein bottle and suppose that $f : M \to M$ is partially hyperbolic.  Since the mapping class group of $M$ is finite (see \cite{McCullough}), after passing to a power of $f$, we can assume that $f$ is isotopic to the identity. Now the diameter of balls in the universal cover of $M$ grow at a polynomial rate under the action of $\widetilde{f}$ since $\pi_1(M)$ is almost nilpotent. However, the arguments above demonstrate that partial hyperbolicity of $f$ is accompanied by exponential growth of the diameter of balls in the universal cover under the action of $\widetilde{f}.$ So there are no partially hyperbolic diffeomorphisms on $M.$
\end{remark}

The example below (from \cite{BoWi}) shows that Theorem~\ref{theorem:nilpotent} need not hold when $\pi_1(M)$ is almost nilpotent.

\begin{example}
Define $\overline{f}: T^3 \to T^3$ as $\bar{f}(v,z)=(Av,z)$, where $v \in T^2$, $z \in S^1$, and $A$ is the induced map on $T^2$ by a hyperbolic matrix in $\textrm{SL}_2(\mathbb{Z}).$ This dynamically coherent partially hyperbolic diffeomorphism commutes with $t: T^3 \to T^3$ defined as $t(v,z) = (-v,z+0.5).$ Since $t$ acts freely on $T^3$, we obtain a dynamically coherent partially hyperbolic diffeomorphism $f$ on $T^3 / t.$

The fundamental group of $T^3 / t$ is almost nilpotent since $T^3$ is its double cover; in fact, $\pi_1(T^3 / t)$ is almost abelian. Also, $H_1(T^3/t ,\mathbb{Z}) = \mathbb{Z} \oplus \mathbb{Z}_2 \oplus \mathbb{Z}_2$, and therefore, $H_1(T^3/t,\mathbb{R}) = \mathbb{R}.$ So the induced action $f_*$ on $H_1(T^3/t,\mathbb{R})$ cannot be partially hyperbolic.
\end{example}

\section{Dynamical coherence} \label{section:dc3}

The arguments in this section are inspired by \cite{BBI2} where the authors show that any strong partially hyperbolic diffeomorphism on $T^3$ is dynamically coherent. We establish dynamical coherence for all strong partially hyperbolic diffeomorphisms on closed $3$-manifolds that have nilpotent fundamental groups.

\subsection{Homological center bunching}

Assume that $f$ is a strong partially hyperbolic diffeomorphism on a closed $3$-manifold $M$ and $\pi_1(M)$ is nilpotent.  So $M$ is a circle bundle over the torus and the induced action of $f_*$ on $H_1(M, \mathbb{R})$ is partially hyperbolic. Let $\lambda^s$, $\lambda^c$, and $\lambda^u$ be the stable, center, and unstable eigenvalues of $f_*.$ If $H_1(M, \mathbb{R})$ is two dimensional, let $\lambda^c = \pm 1$, depending on whether $f$ is isotopic to a fiber preserving map that preserves or reverses the orientation on the fiber. 

\begin{definition}
Let $M$ be a circle bundle over the torus ($T^2$) that supports a strong partially hyperbolic diffeomorphism $f.$  In particular, there exists constants $\nu$, $\hat{\nu}$, $\gamma$, and $\hat{\gamma}$ that satisfy the inequalities in Definition~\ref{definition:phype}. The strong partially hyperbolic diffeomorphism $f$ satisfies the \textit{homological center bunching} condition if 
\[
\abs{\lambda^s}  \leq \nu <  \gamma \leq \abs{\lambda^c}  \leq  \hat{\gamma}^{-1}< \hat{\nu}^{-1} \leq \abs{\lambda^u},
\]
where $\lambda^s$, $\lambda^c$, and $\lambda^u$ are the stable, center, and unstable eigenvalues of $f_*.$
\end{definition}

We will not attempt to recreate the argument on $T^3$ in \cite{BBI2}, especially since the proof is simpler when all the eigenvalues are real and $\abs{\lambda^c} = 1.$ This is exactly the situation when $M$ is a non-trivial circle bundle over the torus, and so, assume that this is the case.  After passing to a power of $f$, we can assume that $f$ is orientation preserving and all the eigenvalues are positive. These assumptions hold throughout this section unless stated otherwise.

When $M$ is a non-trivial circle bundle over the torus, $f$ is isotopic to a fiber preserving homeomorphism and the induced map on the base torus $q: T^2 \to T^2$ is isotopic to a linear Anosov diffeomorphism. So we can isotope $f$ to a fiber preserving map $F$ such that the induced homeomorphism by $F$ on the base is the linear Anosov diffeomorphism $q_*.$ Let $F'$ be a partially hyperbolic diffeomorphism on $M$ just like the one constructed in Example~\ref{example:skew}, with the induced action on the base torus equal to $q_*$ and the identity along the circular fibers. Now $F$ is perhaps not isotopic to $F'$, but $F$ preserves the center circular fibers associated with $F'.$ In fact the action of $F$ on the circular fibers is the same as the action of $F'.$ This implies that the lift $\widetilde{F}$ to the universal cover has the same action on the center leaves of $\widetilde{F}'$, the lift of $F'$ to the universal cover. Also, the action of $\widetilde{F}$ on the lifts of the center-stable and the center-unstable leaves associated with $\widetilde{F}'$ is the same as the action of $\widetilde{F}'.$

Let $\F^{cs}_{\epsilon}$ and $\F^{cu}_{\epsilon}$ be the foliations on the appropriate finite cover given by Theorem~\ref{theorem:epsilon}, with $\epsilon$ sufficiently small so that these foliations are Reebless, the lift of $\F^u$ (the unstable foliation of $f$) to this finite cover is transverse to $\F^{cs}_{\epsilon}$, and the lift of $\F^s$ (the stable foliation of $f$) is transverse to $\F^{cu}_{\epsilon}.$ Let $\F^{cs}_F$, $\F^{cu}_F$, $\F^u_F$, $\F^s_F$, and $\F^c_F$ be the foliations associated with $F$ (these foliations are actually associated with $F'$). We will use tildes to denote the lifts of all these objects to the universal cover of $M.$ In particular, $\widetilde{F}$ is the lift of $F$ isotopic to the lift $\widetilde{f}$ of $f.$

We will now establish homological center bunching under the conditions described above---$M$ is a non-trivial circle bundle over the torus and all the eigenvalues of $f_*$ are positive. This is essentially the only case we are interested in since dynamical coherence for strong partially hyperbolic diffeomorphisms on $T^3$ has been demonstrated in \cite{BBI2}.

\begin{lemma} \label{lemma:half}
Let $f$ and $M$ be as described above. The diffeomorphism $f$ satisfies the homological center bunching condition.
\end{lemma}
\begin{proof}
When $M$ is a non-trivial circle bundle over the torus, $\lambda^c = 1.$  If $\gamma >1$ or $\hat{\gamma}^{-1} < 1$, then $f$ is Anosov, which implies that $M$ must be $T^3.$ Since this is not the case, we have $\gamma \leq \lambda^c \leq \hat{\gamma}^{-1}.$ 

For the map $\widetilde{F}$, points move apart at the rate of $\lambda^u$ in the unstable direction for $\widetilde{F}.$ Also, points move apart  in the center-stable directions at a rate that is at most linear under the action of $\widetilde{F}.$  Now consider a tube (cylinder) with the disks parallel to the leaves in $\widetilde{\F}^{cs}_F$ and height along the unstable direction of $\widetilde{F}.$ Since $\widetilde{f}$ is isotopic to $\widetilde{F}$, the volume of this tube containing a short segment $l$ of the unstable manifold of $\widetilde{f}$ under the action of $\widetilde{f}^n$ is bounded above by $k n^3 (\lambda^u)^n$, for some constant $k>0$; we use $k n^3 (\lambda^u)^n$ instead of $k n^2 (\lambda^u)^n$ since areas can grow cubically on Nil manifolds (see page 49 of \cite{Weinberger} for instance). Now recall that $\widetilde{f}^n(l)$ is not permitted to limit on itself. If two points on $\widetilde{f}^n(l)$ are sufficiently close, we may close up a segment in $\widetilde{f}^n(l)$ and obtain a closed curve transverse to $\widetilde{\F}^{cs}_{\epsilon}.$ This closed curve projects to a null-homotopic closed curve on $M$ that is transverse to $\F^{cs}_{\epsilon}$ and this is impossible since $\F^{cs}_{\epsilon}$ contains no Reeb components. So there exists a fixed $\delta >0$ such that for each $n>0$, there exists a cylindrical tube with axis $\widetilde{f}^n(l)$ and radius $\delta$ that fits inside a cylinder with volume bounded above by $k n^3 (\lambda^u)^n.$ Note that the tube with axis $\widetilde{f}^n(l)$ and radius $\delta$ does not intersect itself and its volume is at least a constant times $(\hat{\nu}^{-1})^n.$ If $\hat{\nu}^{-1} > \lambda^u$, we obtain a contradiction for $n$ sufficiently large.  So $\hat{\gamma}^{-1} < \hat{\nu}^{-1} \leq \lambda^u.$ A completely symmetric argument provides the other inequality $\lambda^s \leq \nu <  \gamma.$ These inequalities establish homological center bunching for $f.$
\end{proof}

\subsection{Quasi-isometry of $\widetilde{\F}^u$ and $\widetilde{\F}^s$}

\begin{definition}
A foliation $\widetilde{\cal{Q}}$ of a simply connected Riemannian manifold $\widetilde{M}$ is quasi-isometric if there are positive constants $a$ and $b$ such that for any two points $x$ and $y$ which lie in the same leaf of $\widetilde{\cal{Q}}$,
\[
d_{\widetilde{\cal{Q}}}(x,y) \leq a \cdot d(x,y) + b.
\] 
\end{definition}

In the present context, $M$ is a $3$-manifold with a partially hyperbolic diffeomorphism and $\widetilde{M}$ is its universal cover.  Also, the metric on $\widetilde{M}$ is the lift of the metric on $M$ in Definition~\ref{definition:phype}.
We will establish that the lifts of unstable and stable foliations are quasi-isometric in the universal cover in Theorem~\ref{theorem:quasi-isometric} below, and then the results in \cite{Brin} will imply that the partially hyperbolic diffeomorphism is dynamically coherent.  First, we prove an important lemma.

\begin{lemma}
Every leaf in $\widetilde{\F}^{cs}_\epsilon$ lies within a  bounded distance of any leaf in $\widetilde{\F}^{cs}_F$ that it intersects.\end{lemma}
\begin{proof}
Let $\widetilde{L}$ be a leaf in $\widetilde{\F}^{cs}_\epsilon$ that does not stay within a bounded distance of some leaf in $\widetilde{\F}^{cs}_F$ that it intersects. Now, there exists a surface $\widetilde{L}^{cs}$ tangent to $\widetilde{E}^{cs}$ that has the same property (since there exists a continuous map $h_\epsilon:\overline{M} \to \overline{M}$ such that $\textrm{dist}_{C^0}(h_\epsilon, id_M) < \epsilon$ and $h_\epsilon$ maps leaves of ${\F}^{cs}_\epsilon$ to surfaces tangent to $E^{cs}$). So one can find a sufficiently long center-stable curve $l$ in the universal cover that does not stay within any given distance of a leaf in $\widetilde{\F}^{cs}_F$ that it intersects. Two points sufficiently far away in the unstable direction of $\widetilde{F}$ on $l$ move apart at the rate of a constant times $(\lambda^u)^n$ under the iteration of $\widetilde{F}$, and these points must move apart at the rate of a constant times $(\lambda^u)^n$ under the iteration of $\widetilde{f}$ also.  However, these points move apart at the rate that is at most a constant times $(\hat{\gamma}^{-1})^n$ under the iteration of $\widetilde{f}.$ Since $\hat{\gamma}^{-1} < \lambda^u$ due to homological center bunching, we obtain a contradiction. So $\widetilde{L}$ stays within a bounded distance of any leaf in $\widetilde{\F}^{cs}_F$ that it intersects.
\end{proof}

A similar argument shows that every leaf in $\widetilde{\F}^{cu}_\epsilon$ lies within a  bounded distance of a leaf in $\widetilde{\F}^{cu}_F$ that it intersects.

\begin{theorem} \label{theorem:quasi-isometric}
When $M$ is a non-trivial circle bundle over the torus, the lifts of the unstable and stable foliations are quasi-isometric in the universal cover.
\end{theorem}
\begin{proof}
Let $T$ be a ``horizontal'' covering translation that corresponds to a ``horizontal'' generator of $\pi_1(M).$ A ``horizontal'' generator on $M$ is any generator not in the direction of the center. Let $\widetilde{L} \in \widetilde{\F}^{cs}_\epsilon$ be a leaf that intersects the leaf $\widetilde{L}_F$ in $\widetilde{\F}^{cs}_F.$  The previous lemma implies that $\widetilde{L}$ lies within a bounded distance $D$ of $\widetilde{L}_F$, and so, $T^n(\widetilde{L})$ is within distance $D$ of $T^n(\widetilde{L}_F)$ for all $n.$

We will first show that there exists a $K>0$ so that every segment of length $K$ in $\widetilde{\F}^u$ intersects at least one of the leaves $T^n(\widetilde{L}).$ If this is not the case, for each $n>0$, there exists a curve $J_n$ in $\widetilde{\F}^u$ of length $1$ such that $\widetilde{f}^n(J_n)$ lies in between two covering translates of $\widetilde{L}.$ In particular, $\widetilde{f}^n(J_n)$ lie within a uniformly bounded distance from the lifts of these center-stable leaves for the map $F$, since the translates of $\widetilde{L}$ lie within a fixed bounded distance from the translates of $\widetilde{L}_F.$  Furthermore, for all $i >0$, $\widetilde{f}^i(J_n)$ also lie within a uniformly bounded distance from the center-unstable leaves for the map $F.$ If this is not so, these curves would grow in length under backward iteration of $F$, and hence, would grow under the backward iteration of $f$, which is a contradiction to the fact that the length of these curves decreases under backward iteration of $f.$ So $\widetilde{f}^n(J_n)$ lie within a uniformly bounded distance from the lifts of the center leaves in the foliation $\widetilde{\F}^c_F.$  This implies that the diameter of $\widetilde{f}^n(J_n)$ is bounded in the stable and unstable directions of $F.$ The diameter of $\widetilde{f}^n(J_n)$ in the center direction of $F$ is bounded above by $kn$, for some constant $k>0.$ So $\widetilde{f}^n(J_n)$ is contained in a ball of volume at most a constant times $n^4$ ($\pi_1(M)$ has quartic growth).  Again, since there are no Reeb components in $\F^{cs}_{\epsilon}$, there exists a properly imbedded cylindrical tube with axis $\widetilde{f}^n(J_n)$ and radius $\delta$ that fits inside this ball of volume at most a constant times $n^4.$ The length of $\widetilde{f}^n(J_n)$ is at least $(\hat{\nu}^{-1})^n$, and so, the volume of the cylindrical tube is at least a constant times $(\hat{\nu}^{-1})^n.$ For $n$ sufficiently large, we obtain a contradiction. So our claim holds.

A curve in $\widetilde{\F}^u$ cannot intersect a leaf in $\widetilde{\F}^{cs}_\epsilon$ twice because of the absence of Reeb components in ${\F}^{cs}_\epsilon.$ So a curve of length $nK$ in $\widetilde{\F}^u$ must intersect $n$ different translates of $\widetilde{L}.$ This shows that the distance between its endpoints is at least a constant times $(n-1)$, which establishes the quasi-isometry of  $\widetilde{\F}^u$ in the universal cover.  
The proof of the fact that  $\widetilde{\F}^s$ is quasi-isometric in the universal cover is similar. 
\end{proof}

\begin{remark}
We actually show that the lifted stable and unstable foliations are quasi-isometric in the universal cover for the Euclidean metric since distances along the horizontal generators are comparable to Euclidean distances when $M$ is a Nil manifold (see page 19 of \cite{coarse}); distances along the vertical direction are comparable to the square root of Euclidean distance on Nil manifolds.
\end{remark}

Note that the authors in \cite{BBI2} show that $\widetilde{\F}^s$ and $\widetilde{\F}^u$ are quasi-isometric in the universal cover when $f$ is a strong partially hyperbolic diffeomorphism on $T^3$---their argument is more complicated. Brin in \cite{Brin} has shown that quasi-isometry of the lifted stable and unstable foliations is sufficient to establish dynamical coherence. These results and Theorem~\ref{theorem:quasi-isometric} imply the following theorems.

\begin{theorem:quasi-isometry}
Let $f$ be a strong partially hyperbolic diffeomorphism on a closed $3$-manifold $M.$ If $\pi_1(M)$ is nilpotent, the lifts of the unstable and stable foliations are quasi-isometric in the universal cover.
\end{theorem:quasi-isometry}

\begin{theorem:dc3}
Let $f$ be a strong partially hyperbolic diffeomorphism on a closed $3$-manifold $M.$ If $\pi_1(M)$ is nilpotent, $f$ is dynamically coherent.
\end{theorem:dc3}

\section{Lipschitz distributions} \label{section:coherence}

In this section, we show that center bunching and Lipschitz distributions imply dynamical coherence. R. Saghin has proved the same result (unpublished) for strong partially hyperbolic diffeomorphisms, using different techniques. Theorem~\ref{theorem:coherence} was also established in \cite{BuWi1} by Burns and Wilkinson under stronger assumptions---they needed the distributions to be $C^2.$ The proof we present is similar to the one in \cite{HHU} for smooth distributions.

The center bunching condition given below is not the (mixed) center bunching condition that appears in \cite{BuWi1}, but this condition is related to the \textit{symmetric} center bunching condition used in \cite{BuWi2}.

\begin{definition}
A partially hyperbolic diffeomorphism $f:M \to M$ is \textit{center bunched} if for every $p \in M$, $\hat{\nu}(p) < \hat{\gamma}(p)^{2}$ and $\nu(p) < \gamma(p)^2.$
\end{definition}

There is an example in \cite{Smale}, attributed to A.~Borel by Smale, of a partially hyperbolic diffeomorphism on a six dimensional manifold that has smooth distributions, is not dynamically coherent, and does not satisfy this center bunching condition ($\hat{\nu} = \hat{\gamma}^{2}$ and $\nu = \gamma^2$ in this example). Borel's example is well known and it also appears in \cite{KH} and \cite{BuWi2}.

The proof of Theorem~\ref{theorem:coherence} relies on a version of Frobenius' theorem for Lipschitz distributions, which we now state. A very famous theorem of Radamacher asserts that Lipschitz functions are differentiable almost everywhere. This allows the Lie bracket and the exterior derivative to be defined almost everywhere, and then the Lipschitz Frobenius' Theorem follows via careful approximations of Lipschitz distributions by smooth ones. This program is carried out by S.~Simi\'{c} in \cite{Simic}.

\begin{theorem}[Simi\'{c}]
Let $E$ be a $k$-dimensional Lipschitz distribution on a compact smooth $n$-dimensional manifold $M.$ If $E$ is involutive almost everywhere, then every point of $M$ has a coordinate neighborhood $(U ; x_1, . . . , x_n)$ such that: 
\begin{enumerate}[\normalfont 1)]
\item Each map $x_i : U \to R$ is Lipschitz.
\item The slices $x_{k+1}=$ constant, . . . , $x_n =$ constant are integral manifolds of $E.$ Moreover, every connected integral manifold of $E$ in $U$ is of class $C^{1,Lip}$ and lies in one of these slices.
\end{enumerate}
Furthermore, through every point $p \in M$ passes a unique maximal connected integral manifold of $E$, and every connected integral manifold of $E$ through $p$ is contained in the maximal one. 
\end{theorem}


We will now introduce notation that will simplify the proof of Theorem~\ref{theorem:coherence}.
If $p \in M$ and $n$ is an integer, then $p_n$ denotes the point $f^n(p).$ Also, for a function $\beta :M \to \mathbb{R}$ and an integer $n>0$, let $\beta_n(p) = \beta(p) \beta(p_1) \dots \beta(p_{n-1}).$

\begin{theorem:coherence}
Let $f$ be a partially hyperbolic diffeomorphism. If $f$ is center bunched and the distributions $E^{cs}$ and $E^{cu}$ are Lipschitz, then $f$ is dynamically coherent.
\end{theorem:coherence}
\begin{proof}
Assume that $\alpha$ is a Lipschitz 1-form such that $\alpha = 0$ for every vector in $E^{cs}$ and let $V$ and $W$ be any two Lipschitz vector fields in $E^{cs}.$  Since the distribution $E^{cs}$ and the vector fields are Lipschitz, there exists a set of full measure for which the exterior derivative of $\alpha$ and the Lie bracket of $V$ and $W$ are defined. For every point in this full measure set,
\begin{align*}
d\alpha (V,W) & = V(\alpha (W)) - W(\alpha (V)) - \alpha ([V,W]) \\
& =  - \alpha([V,W]).
\end{align*}
We will use this observation to prove that $E^{cs}$ is involutive almost everywhere, which is sufficient to prove that $E^{cs}$ is integrable. A similar argument will establish that $E^{cu}$ is integrable.

Now let $W^{cs}$ and $V^{cs}$ be two Lipschitz vector fields in $E^{cs}$, and then, let $p$ be a point at which $E^{cs}$ is differentiable and $[V^{cs}, W^{cs}]_p$ is defined. If $[V^{cs},W^{cs}]_p \in E^{cs}(p)$, there is nothing to prove. Assume that the projection of $[V^{cs},W^{cs}]_p $ onto $E^u(p)$ is nontrivial, and let $v^u \in E^u(p)$ be this nontrivial projection.  We will first define a 1-form $\alpha$ in a sufficiently small neighborhood $N$ of a limit point $l$ in $\omega(p).$ There exists a sequence $\{n_k \}$ of positive integers such that $f^{n_k}(p) \to l$ and $f^{n_k}(p)$ are all in $N.$ There exists a subsequence $\{n_{k_m}\}$ of $\{n_k \}$ such that vectors $Df^{n_{k_m}}(v^u)/ \p Df^{n_{k_m}}(v^u) \p$ converge to a vector $v$ at $l.$ Define a vector field $X$ in $N$ by choosing a smooth extension of the vector $v$ at $l.$ Now let $\alpha = 1$ on $X$ and zero on the orthogonal complement of $X$ in $N$. This locally defined form can be extend to entire manifold so that $\alpha$ is zero outside a neighborhood that contains $N.$ Note that since $f$ is a $C^1$ diffeomorphism, the distributions are differentiable along the orbit of $p$ in $N$, and therefore, $d\alpha$ is defined along the orbit of $p$ in $N.$ Furthermore, since $\alpha$ is Lipschitz, there exists a constant $K$ such that where $d\alpha$ is defined, $|d\alpha (x,y)| < K$ for any unit vectors $x$ and $y.$

Now there exists a $n_{k_r}$ so that $\alpha(Df^{n_{k_m}}(v^u)) > 0$ for all $n_{k_m} > n_{k_r}.$ Let $U=Df^{n_{k_r}}(v^u)$, let $V= f^{n_{k_r}}_*(V^{cs})$, let $W=f^{n_{k_r}}_*(W^{cs})$, and let $p' = f^{n_{k_r}}(p).$ To simplify notation, we will supress all these subscripts and assume that  $\{n_{k_m}\}_{m=r}^{\infty}$ is the increasing sequence of positive integers.
The form $\alpha$ has been chosen to guarantee that  for sufficiently large $n$, $|f^{n*} \alpha(U)| \geq k_1 \hat{\nu}^{-1}_{n}(p') |\alpha(U)|$, where $0 < k_1 \leq 1.$ 
Also, for $n>0$ we have
\begin{align*}
|f^{n*} \alpha(U)| & = |f^{n*} \alpha([V,W]_{p'})| = |f^{n*} d \alpha(W_{p'},V_{p'})|\\
& = | d\alpha (f^n_* W_{p'}, f^n_* V_{p'}) | \leq k_2 \hat{\gamma}_{n}^{-1}(p')^2,
\end{align*}
where $k_2 > 0$ is a constant. Therefore, for all $n$ sufficiently large,
\begin{equation*}
k_1 \hat{\nu}^{-1}_{n}(p') |\alpha(U)| \leq |f^{n*} \alpha(U)| \leq k_2 \hat{\gamma}_{n}^{-1}(p')^2.
\end{equation*}
This is impossible since $\hat{\nu}(q) < \hat{\gamma}(q)^{2}$ for all $q \in M.$
Therefore, $[V^{cs},W^{cs}]_{p} \in E^{cs}(p).$ This argument shows that $E^{cs}$ is involutive almost everywhere.

Note that we do not need the center bunching condition when $E^c$ is one-dimensional. To see this, let $W^{cs}$ be a Lipschitz vector filed in $E^{cs}$ and let $V^{s}$ be a Lipschitz vector field in $E^s.$ If $[V^s, W^{cs}]_{p}$ is not in $E^{cs}(p)$, construct the 1-form $\alpha$ as before. Using notation similar to the notation above, for all $n$ sufficiently large we have
\begin{align*}
k_1 \hat{\nu}^{-1}_{n}(p') |\alpha(U)| & \leq |f^{n*} \alpha(U)| = |f^{n*} d \alpha(W_{p'},V_{p'})|\\
& = | d\alpha (f^n_* W_{p'}, f^n_* V_{p'}) | \leq k_2  \hat{\gamma}_{n}^{-1}(p') \nu_{n}(p'),
\end{align*}
for some constant $k_2 > 0.$ This is impossible since $\hat{\gamma}^{-1}(q) < \hat{\nu}^{-1}(q)$ and $\nu(q) < 1$ for all $q \in M.$ So $[V^s, W^{cs}]_p \in E^{cs}(p).$

These arguments imply that $E^{cs}$ is integrable. A similar argument, which uses the fact that $\nu(q) < \gamma(q)^2$ for all $q \in M$, establishes that $E^{cu}$ is also integrable. Again, if $E^c$ is one-dimensional, the center bunching condition $\nu(q) < \gamma(q)^2$ for all $q \in M$ is not required to show that $E^{cu}$ is integrable.
\end{proof}

In the proof above, the center bunching conditions  $\hat{\nu} < \hat{\gamma}^{2}$ and $\nu < \gamma^2$ were only required when $\textrm{dim}(E^c) > 1.$

\begin{corollary} \label{corollary:one_dim_center}
Let $f$ be a partially hyperbolic diffeomorphism with a one-dimensional center bundle. If the distributions $E^{cs}$ and $E^{cu}$ are Lipschitz, then $f$ is dynamically coherent. 
\end{corollary}

In particular, if $f$ is a partially hyperbolic diffeomorphism on a $3$-manifold and if the distributions $E^{cs}$ and $E^{cu}$ are Lipschitz, then $f$ is dynamically coherent. Also note that Corollary~\ref{corollary:one_dim_center} with the hypothesis that $E^c$ is Lipschitz is a direct consequence of Proposition 2.7  in \cite{BuWi2}.

\section{Questions} \label{section:questions}

We end by presenting several open questions. Some of these appear above in this paper; we collect all questions in this section for the benefit of the reader.

\begin{question}
Which $3$-manifolds support partially hyperbolic diffeomorphisms?
\end{question}

\begin{question}
Does there exist a $3$-manifold with exponential growth in its fundamental group that supports a partially hyperbolic diffeomorphism but does not support an Anosov flow?
\end{question}

Very recently, F. Rodriguez Hertz, M.A. Rodriguez Hertz, and R. Ures announced the existence of a non-dynamically coherent partially hyperbolic diffeomorphism of ${T}^3$ (homotopic to a skew product). This example cannot be of a strong partially hyperbolic diffeomorphism.
\begin{question}
Is every strong partially hyperbolic diffeomorphism on a $3$-manifold dynamically coherent?
\end{question}

\begin{question}[Burns and Wilkinson \cite{BuWi2}]
Let $M$ be a $n$-dimensional manifold that supports a partially hyperbolic diffeomorphism $f.$ If $f$ is (symmetric) center bunched, is $f$ necessarily dynamically coherent?
\end{question}

In other words, prove Theorem~\ref{theorem:coherence} without assuming that the distributions $E^{cs}$ and $E^{cu}$ are Lipschitz.

\begin{question} \label{question:universal}
Let $M$ be an $n$-dimensional manifold that supports a partially hyperbolic diffeomorphism. If the distributions $E^c$ and $E^u$ are both one-dimensional, then is the universal cover of $M$ homeomorphic to $\mathbb{R}^n$?
\end{question}

The condition $\textrm{dim}(E^c) = 1$ is necessary since we may easily construct a partially hyperbolic diffeomorphism on $T^2 \times S^2$, by taking a product of an Anosov diffeomorphism with the identity, so that $\textrm{dim}(E^c) = 2$ and the universal cover is $\mathbb{R}^2 \times S^2.$
When $f$ is dynamically coherent, $\textrm{dim}(E^u) = 1$ implies that $\F^{cs}$ is a codimension one foliation. For instance, a time-one map of a codimension one Anosov flow ($\textrm{dim}(E^u)=1$) is dynamically coherent and has one-dimensional center and unstable distributions. It is well known that the universal cover of a manifold that supports a codimension one Anosov flow is homeomorphic to $\mathbb{R}^n$, and Question~\ref{question:universal} asks whether the same is true for a manifold that supports a \textit{codimension one partially hyperbolic diffeomorphism} ($\textrm{dim}(E^c) = 1$ and $\textrm{dim}(E^u)=1$). It is not too difficult to prove that this is indeed the case under the additional hypotheses of dynamical coherence and no compressible leaves in $\F^{cs}.$ 

\begin{question}
Let $M$ be a $n$-dimensional manifold that supports a codimension one partially hyperbolic diffeomorphism $f.$ If $f$ is dynamically coherent, then is every closed leaf in $\F^{cs}$ incompressible?
\end{question}

\end{document}